\documentclass[12pt, 14paper,reqno]{amsart}
\usepackage{amsmath,amsfonts,amssymb}
\usepackage{mathrsfs}
\usepackage{longtable}
\usepackage[breaklinks]{hyperref}
\usepackage{graphicx}
\makeatletter
\@namedef{subjclassname@2020}{%
  \textup{2020} Mathematics Subject Classification}
\makeatother
\theoremstyle{plain}
\newtheorem{thm}{Theorem}[section]
\newtheorem{lem}{Lemma}[section]

\newtheorem{prop}{Proposition}[section]

\newtheorem{conj}{Conjecture}[section]

\newtheorem{thma}{Theorem}

\theoremstyle{proof}
\numberwithin{equation}{section}

\newcommand \QQ {{\mathbb Q}}
\newcommand \ZZ {{\mathbb Z}}

\begin{document} 
\title[simultaneous $3$-divisibility of class numbers ]{On the simultaneous $3$-divisibility of class numbers of quadruples of real quadratic fields}
\author{Kalyan Banerjee, Ankurjyoti Chutia and Azizul Hoque}
%\address{Harish-Chandra Research Institute, Chhatnag road, Jhunsi, Allahabad-211019, India}
\address{KB @Department of Mathematics, SRM University AP, Mangalagiri-Mandal, Amaravati-522240, Andhra Pradesh, India}
\email{kalyan.b@srmap.edu.in}
\address{AC @Department of Mathematics, Faculty of Science, Gauhati University, Guwahati-781014, Assam, India}
\email{ankurjyoti878@gmail.com}
\address{AH @Department of Mathematics, Faculty of Science, Rangapara College, Rangapara, Sonitpur-784505, Assam, India}
\email{ahoque.ms@gmail.com}


\date{\today}


\begin{abstract}
In this paper, we construct infinitely many quadruples of real quadratic fields whose class numbers are all divisible by $3$. To the best of our knowledge, this is the first result towards the divisibility of the class numbers of certain tuples of real quadratic fields. At the end, we give an application of this result to produce some elliptic curves having a $3$-torsion subgroup. 
\end{abstract} 
\maketitle{}
%\tableofcontents
\section{Introduction}
It was conjectured by Gauss that there are infinitely many real quadratic fields with class number one, which is still open. In contrast, the situation for imaginary quadratic fields is completely understood: only nine such fields have class number one, and all imaginary quadratic fields with class numbers up to $100$ have been classified (cf. \cite{WA2004}). These results highlight the importance of understanding the arithmetic of class numbers. In particular, their divisibility properties shed light on the structure of the associated ideal class groups. A striking development in this direction is the proof that, for any positive integer $n$, there exist infinitely many real (resp. imaginary) quadratic fields whose class numbers are divisible by $n$ (see, \cite{AC1955, CHYP18, H2022, YA1970}). Moreover,  there is a “Spiegelungssatz” due to Scholz \cite{SC32}, which relates the ideal class group of a real quadratic field to that of an imaginary quadratic field.  As a consequence of this Spiegelungssatz, we can deduce that if $3$ divides the class number of a real quadratic field $\QQ(\sqrt{d})$, then $3$ also divides the class number of the imaginary quadratic field $\QQ(\sqrt{-3d})$. Komatsu \cite{KO02} was motivated by this consequence, who proved the existence of an infinite family of pairs of quadratic fields of the form $\QQ(\sqrt{d})$ and $\QQ(\sqrt{md})$ with $m, d\in \mathbb{Z}$ whose class numbers are divisible by $3$. He further extended this result in \cite{KO17} to the $n$-divisibility of the class numbers of pairs of imaginary quadratic fields of the above form. Later, Iizuka \cite{IZ18} considered an analogous problem, and proved the existence of infinitely many pairs of  imaginary quadratic fields of the form $\mathbb{Q}(\sqrt{d})$  and $\mathbb{Q}(\sqrt{d+1})$ with $d\in\mathbb{Z}$ whose class numbers are all divisible by $3$.  This helped him to frame the following conjecture. 

\begin{conj}[{\cite[126 p.]{IZ18}, Conjecture}]\label{conjIZ}
For any prime number $p$ and any positive integer $m$, there is an infinite family of $m + 1$ successive real (or imaginary) quadratic fields, 
$$\mathbb{Q}(\sqrt{d}), \mathbb{Q}(\sqrt{d+1}), \cdots, \mathbb{Q}(\sqrt{d+m})$$
with $d\in \mathbb{Z}$ whose class numbers are all divisible by $p$.
\end{conj}
In \cite{CM21}, Chattopadhyay and Muthukrishnan  extended the result of Iizuka \cite[Theorem 1]{IZ18} by proving the $3$-divisibility of the class numbers of an infinite family of triples of imaginary quadratic fields. Their proof is based on the construction of an unramified, cyclic cubic extension of a quadratic field and the classical Spiegelungssatz of Scholz. Krishnamoorthy and Pasupulati \cite{KP21} further extended Iizuka's result \cite[Theorem 1]{IZ18} from $3$-divisibility to $p$-divisibility for any prime $p$. In particular, it resolved Conjecture \ref{conjIZ} when $m=1$. In  \cite{XC20}, Xie and Chao  proved that there are infinitely many pairs of imaginary quadratic fields of the form $\QQ(\sqrt{d})$ and $\QQ(\sqrt{d+m})$, whose class groups have an element of order $n$ respectively. They used Yamamoto’s \cite{YA1970} construction to prove this result. The third author \cite{AH2022} constructed an infinite family of quintuples of imaginary quadratic fields of the form  $\QQ(\sqrt{d}), \QQ(\sqrt{d+1}), \QQ(\sqrt{d+4}), \QQ(\sqrt{d+36})$ and $ \QQ(\sqrt{d+100})$ whose class numbers are all divisible by a given odd integer $n\geq 3$. This result helped the author to give a complete proof of Conjecture \ref{conjIZ} in a more general case, when $m = 1$. Analogously, it gives an affirmative answer to a weaker version of Conjecture \ref{conjIZ} for $m \geq 3$. Chakraborty and the third  author  \cite{CH2023} constructed an infinite family of certain tuples of imaginary quadratic fields of the form $\QQ(\sqrt{d}), \QQ(\sqrt{d+1}), \QQ(\sqrt{4d+1})$ and $\QQ(\sqrt{2d+4^m})$ with $d, m\in \ZZ$ and $1\leq m\leq 2|d|$ satisfying the $n$-divisibility of their class numbers for a given odd integer $n\geq 3$. In the spirit of Conjecture \ref{conjIZ}, the third author \cite{AH2024} considered the problem of $n$-divisibility of class numbers of the tuples of imaginary quadratic fields of the form, 
$$\left(\QQ(\sqrt{d}), \QQ(\sqrt{d+1}), \QQ(\sqrt{2(d+2)}), \QQ(\sqrt{3(d+3)}), \cdots, \QQ(\sqrt{m(d+m)})\right)$$ with $d\in \ZZ$ for a given integer $m\geq 1$. 
 In a similar spirit, we construct an infinite family of quadruples of real quadratic fields whose class numbers are all divisible by $3$. The precise result in this paper is the following:

\begin{thm}\label{thm} There are infinitely many quadruples of real quadratic fields of the form,

\begin{align*}& \big(\mathbb{Q}(\sqrt{D}), \mathbb{Q}(\sqrt{216000D^3+457200D^2+322580D+75866}), \mathbb{Q}(\sqrt{432D^3+1080D^2+900D+223}), \\
& \mathbb{Q}(\sqrt{40500D^3+89100D^2+65340D+16215}) \big)\end{align*} with $D\in \mathbb{N}$ whose class numbers are all divisible by $3$. 
\end{thm}
\section{Construction of cyclic, cubic and unramified extensions}
We recall the following beautiful result of Kishi and Miyake \cite{KM00} to construct an unramified, cyclic cubic extension of a given quadratic field.  
\begin{thma}[{\cite[Main Theorem]{KM00}}]\label{KMtheorem}
For any two integers $u$ and $v$, let
 \begin{equation}\label{eq2.1}
F_{u, v}(Z)=Z^3-uvZ-u^2.
\end{equation}
If 
\begin{enumerate}\label{ctn3.1}
\item[(a)] $u$ and $v$ are relatively prime;
\item[(b)] $F_{u, v}(Z)$ is irreducible over $\mathbb{Q}$;
\item[(c)] the discriminant $D_{F_{u, v}}$ of $F_{u, v}(Z)$ is not a perfect square in $\mathbb{Z}$;
\item[(d)] one of the following conditions holds:
\begin{enumerate}
\item[(d.1)] $3\nmid v,$
\item[(d.2)] $3\mid v,\hspace*{3mm} uv\not\equiv 3\pmod 9, \hspace*{3mm} u\equiv v\pm 1 \pmod 9 ,$
\item[(d.3)] $3\mid v, \hspace*{3mm} uv\equiv 3 \pmod 9, \hspace*{3mm} u\equiv v\pm 1 \pmod { 27} ,$
\end{enumerate}
\end{enumerate}
then the normal closure of $\mathbb{Q}(\alpha)$, where $\alpha$ is a root of $F_{u,v}(Z)$, is a cyclic, cubic, unramified extension of  $\mathbb{K}=\mathbb{Q}(\sqrt{D_{F_{u,v}}})$; in particular, $\mathbb{K}$ has class number divisible by $3$.
Conversely, every quadratic number field $\mathbb{K}$ with class number divisible by $3$ and every unramified, cyclic and cubic extension of $\mathbb{K}$ is given by suitable choices of integers $u$ and $v$.
\end{thma}

Another approach to construct a cyclic, cubic and unramified extension of a quadratic field is due to Kishi \cite{KI00}. This approach is based on a cubic  polynomial defined by an algebraic integer.  Let $\mathbb{K}$ be a quadratic field and $\mathcal{O}_\mathbb{K}$ its ring of integers. Assume that $\alpha \in \mathcal{O}_\mathbb{K}$ with $N_{\mathbb{K}/\mathbb{Q}}(\alpha)\in \mathbb{Z}^3$. We now define 
\begin{equation*}\label{eq1}
 P_\alpha(X):=X^3-3[N_{\mathbb{K}/\mathbb{Q}} (\alpha)]^{1/3}X-T_{\mathbb{K}/\mathbb{Q}}(\alpha).
\end{equation*}
In \cite{KI98}, Kishi deduced the following criterion for the irreducibility of $P_\alpha(X)$ over $\mathbb{Q}$.
\begin{lem}\label{lemK}
 Let $\mathbb{K}=\mathbb{Q}(\sqrt{d}).$ Suppose $\alpha=\frac{a+b\sqrt{d}}{2} \in {\mathcal{O}}_\mathbb{K}$ 
 with
 $N_{\mathbb{K}/\mathbb{Q}}(\alpha)$ is a cube in $\mathbb{Z}$. Then $P_{\alpha} (X)$  
 is 
 reducible over $\mathbb{Q}$ if and only if $\alpha$ is a cube in $\mathbb{K}$. 
\end{lem}

Let $d$ be a square-free integer other than $1$ and $-3$. Set
$$
\begin{displaystyle}
D:= 
\begin{cases}
-d/3 \hspace*{2mm}\text{if } 3\mid d, \\
  -3d \hspace*{4mm}\text{otherwise}. 
  \end{cases}
  \end{displaystyle}
$$
Let 
$\mathbb{K}=\mathbb{Q}(\sqrt{d})$ and $\mathbb{L}=\mathbb{Q}(\sqrt{D})$. We define,
$$
R_d:=\{\alpha \in\mathcal{O}_\mathbb{K}: \alpha \text{ is not  a  cube in } \mathbb{K}, \text{ but} \ N_{\mathbb{K}/\mathbb{Q}}(\alpha)\text{ is a  cube in } \mathbb{Z}\}
$$
and
$$
R_D:=\{\alpha \in\mathcal{O}_\mathbb{L}: \alpha \text{ is not a cube in } \mathbb{L} \text{ and } N_{\mathbb{L}/\mathbb{Q}}(\alpha)\text{ is a cube in } \mathbb{Z}\}.
$$
It is very clear that the subset $R_d$ (resp. $R_D$) contains all those units in $\mathbb{K}$ which are not cubes in $\mathbb{K}$ (resp. in $\mathbb{L}$). Further assume that
$$
R^*_d:=\{\alpha\in R_d: \gcd(N_{\mathbb{K}/\mathbb{Q}}(\alpha), T_{\mathbb{K}/\mathbb{Q}}(\alpha))=1\}
$$
and 
$$
R^*_D:=\{\alpha\in R_D: \gcd(N_{\mathbb{L}/\mathbb{Q}}(\alpha), T_{\mathbb{L}/\mathbb{Q}}(\alpha))=1\}.
$$

With the help of elements in $R^*_d$ (resp. $R^*_D$), one can construct a cyclic, cubic and unramified extension of $\mathbb{K}$ (resp. $\mathbb{L}$). Kishi used this idea to construct such unramified extensions except at $3$. More precisely, he proved the following:  
\begin{thma}[{\cite[Proposition 6.5]{KI00}}]\label{thmB}
 Let $\alpha \in R^*_D$ (resp. $\alpha \in R^*_d$). Then the splitting field, $S_{\mathbb{Q}}(P_\alpha)$ of $P_\alpha(X)$ over $\mathbb{Q}$ is an $S_3$-field containing $\mathbb{K}=\mathbb{Q}(\sqrt{d})$ (resp. $\mathbb{L}=\mathbb{Q}(\sqrt{D})$) which is a cyclic cubic extension of $\mathbb{K}$ (resp. $\mathbb{L}$) unramified outside $3$ and contains a cubic subfield $\mathbb{K}'$ with $\mathit{v}_3(\Delta_{\mathbb{K}'})\ne 5.$ Conversely, every $S_3$-field containing $\mathbb{K}$ (resp. $\mathbb{L}$) which is unramified outside $3$ over $\mathbb{K}$ (resp. $\mathbb{L}$) and contains a cubic subfield $\mathbb{K}'$ satisfying $\mathit{v}_3(\Delta_{\mathbb{K}'})\ne 5$ is given by $S_{\mathbb{Q}}(P_\alpha)$ with $\alpha \in R^*_D$ (resp. $\alpha \in R^*_d$). 
 \end{thma}

The above theorem helps us to construct cyclic, cubic and unramified extensions outside $3$. Thus, we need to handle the prime $3$, and the following result comes here to rescue us in this case. This result is a particular case of \cite[Theorem 1]{LN83}.
 \begin{lem}\label{lemr3}
  Suppose that
  $$
  g(X):= X^3-aX-b\in \mathbb{Z}[X]
  $$
 is irreducible over $\mathbb{Q}$ and that either $\mathit{v}_3(a)<2$ or $\mathit{v}_3(b)<3$ holds. Let $\theta$ be a root of $g(X).$ 
 Then $3$ is totally ramified in $\mathbb{Q}(\theta)/\mathbb{Q}$ if and only if one of the following conditions holds:
 \begin{enumerate}
  \item[(i)] $1\leq \mathit{v}_3(b)\leq \mathit{v}_3(a),$
  \item[(ii)] $3\mid a, \ a\not\equiv3 \pmod 9, \ 3\nmid b \ and \ b^2\not\equiv a+1 \pmod 9,$
  \item[(iii)] $ a \equiv3 \pmod 9, \ 3\nmid b \ and \ b^2\not\equiv a+1 \pmod {27}.$
 \end{enumerate}
 \end{lem}

\section{Some families of real quadratic fields with class number divisible by $3$}
\begin{prop}\label{ld}
For any positive integer $x$, the class number of the real quadratic field $\mathbb{Q}(\sqrt{216000x^3+457200x^2+322580x+75866})$ is divisible by $3$. 
\end{prop}
\begin{proof} Let us choose $u=4$ and $v=3(180x+127)$. Then $\gcd(u, v)=1$. We set:
$$F_{u,v}(Z):=Z^3-12(180x+127)Z-16.$$ 
Now, reading $F_{u,v}(Z)$ under modulo $5$, we get $F_{u,v}(Z)\equiv Z^3-4Z-1\pmod 5$. It is easy to check that $F_{u,v}(Z)\pmod 5$ is irreducible, and thus $F_{u,v}(Z)$ is irreducible as a polynomial with integer coefficients as well. 

The discriminant of $F_{u,v}(Z)$ is $$\Delta_{F_{u, v}}=4(12(180x+127))^3-27\cdot 16^2.$$  This can be simplified as 
$\Delta_{F_{u, v}}=12^4D$, where 
$D=216000x^3+457200x^2+322580x+75866$. Since $D\equiv 2\pmod 4$, $D$ is not a square in $\mathbb{Z}$ and so is $\Delta_{F_{u,v}}$.

We see that $3\mid v$, and $uv\equiv 3\pmod 9$. Further, $v+1\equiv 4\pmod {27}$. Therefore, $F_{u,v}(Z)$ satisfies the conditions (a)-(c) and (d.3) of Theorem \ref{KMtheorem}. This completes the proof by Theorem \ref{KMtheorem}. 
\end{proof}
\begin{prop}\label{sd}
For any positive integer $y$, the class number of the real quadratic field $\mathbb{Q}(\sqrt{432y^3+1080y^2+900y+223})$ is divisible by $3$. 
\end{prop}
\begin{proof} The proof of this proposition is very similar to that of Proposition \ref{ld}. Nevertheless, we give the proof in brief for the sake of completeness. 

We put $u=2$ and $v=6y+5$. Then $\gcd(u, v)=1$, and we define, $$F_{u,v}(Z):=Z^3-2(6y+5)Z-4. $$
The discriminant of $F_{u,v}(Z)$ is $$\Delta_{F_{u, v}}=4(2(6y+5))^3-27\cdot 4^2,$$ which can be simplified as 
$4^2d$ with $d=432y^3+1080y^2+900y+223$. As $d\equiv 2\pmod 3$, $d$ is not a square in $\mathbb{Z}$ and so is $\Delta_{F_{u, v}}$. 
It is easy to see that $F_{u,v}(Z)\pmod 3$ is irreducible, and thus $F_{u,v}(Z)$ is irreducible over $\mathbb{Z}$ too.  
Since $3\nmid v$, therefore by Theorem \ref{KMtheorem}, we conclude the proof. 
\end{proof}
  In the next proposition, we will use the second method to construct unramified, cyclic cubic extension of a quadratic field. To establish the $3$-divisibility of the class number, we need Hilbert class field. Therefore the for shake of completeness,  we briefly recall the Hilbert class field. The Hilbert class field of a number field  $\mathbb{K}$, denoted by $H(\mathbb{K})$, is defined as the maximal unramified abelian extension of $\mathbb{K}$, which contains all other unramified abelian extensions of $\mathbb{K}$.
\begin{thma}[{\cite[Theorem 5.23 (Artin Reciprocity)]{cox}}]\label{CFTthma}
Let $\mathbb{K}$ be a number field with its group of fractional ideals and class group respectively, $\mathcal{I}_\mathbb{K}$ and $Cl(\mathbb{K})$. If  $H(\mathbb{K})$ is the Hilbert class field of  $\mathbb{K}$, then the Artin map
$$\Phi:~~\mathcal{I}_\mathbb{K}\longrightarrow\mathrm{Gal}(H(\mathbb{K})/\mathbb{K})$$
is surjective, and its kernel is exactly the subgroup $P_\mathbb{K}$ of principal fractional ideals. Thus this map induces the following deep correspondence
$$\mathrm{Gal}(H(\mathbb{K})/\mathbb{K})\cong Cl(\mathbb{K}).$$ 
\end{thma}

\begin{prop}\label{ab}
For any positive integer $k$,  the class number of the real quadratic field $\mathbb{Q}(\sqrt{40500k^3+89100k^2+65340k+16215})$ is divisible by $3$. 
\end{prop}
\begin{proof}
Assume that $d=-(13500k^3+29700k^2+21780k+5405)$. Then $d\not \equiv 0\pmod 3$, and thus we set $D:=-3d=40500k^3+89100k^2+65340k+16215$. As $d\equiv 2\pmod 3$ and $D\equiv 3\pmod 4$, so that both $d$ and $D$ are not perfect squares.  

Let $\mathbb{K}=\mathbb{Q}(\sqrt{-d})$, and define $\alpha \in \mathbb{K}$ as 
$$\alpha:=\frac{9+\sqrt{-d}}{2}.$$
Then $T_{\mathbb{K}/\mathbb{Q}}(\alpha)=9$ and $N_{\mathbb{K}/\mathbb{Q}}(\alpha)=-(15k+11)^3$, and thus  $\gcd(T_{\mathbb{K}/\mathbb{Q}}(\alpha), N_{\mathbb{K}/\mathbb{Q}}(\alpha))=1$. Therefore with this $\alpha$, we define
$$f_\alpha(Z):= Z^3-3\left(N_{\mathbb{K}/\mathbb{Q}}(\alpha)\right)^{1/3}Z-T_{\mathbb{K}/\mathbb{Q}}(\alpha),$$
which is $$f_\alpha(Z)=Z^3+3(15k+11)Z-9.$$
If $Z=p/q$ is a root of $f_\alpha(Z)$ in $\mathbb{Q}$, then $p\mid 9$ and $q\mid 1$, which further imply that 
$$Z=\pm 1, \pm  3, \pm 9.$$
However,  $f_\alpha(\pm1)\ne 0$, $f_\alpha(\pm3)\ne 0$ and $f_\alpha(\pm9)\ne 0$. This confirms that 
$f_\alpha(Z)$ is irreducible over $\mathbb{Q}$. Therefore by Lemma \ref{lemK}, $\alpha$ is not a cube in $\mathbb{K}$, and hence $\alpha \in R_d$. As  $\gcd(T_{\mathbb{K}/\mathbb{Q}}(\alpha), N_{\mathbb{K}/\mathbb{Q}}(\alpha))=1$, so that $\alpha\in R^*_d$. Therefore by Theorem \ref{thmB}, $S_\mathbb{Q}(f_\alpha)$  is a cyclic cubic extension of $\mathbb{L}$ which is unramified outside $3$. 

It remains to check the unramification of $S_\mathbb{Q}(f_\alpha)$ at $3$. To see this, we will apply Lemma \ref{lemr3}. By the assumptions, we see that $v_3(3(15k+11))=1$ and $v_3(b)=v_3(9)=2$. Thus (i) of Lemma \ref{lemr3} does not hold. Rest of two conditions of Lemma \ref{lemr3} do not hold since $b=9\equiv 0\pmod 3$. Thus by Lemma \ref{lemr3}, we can conclude that $S_\mathbb{Q}(f_\alpha)$ is unramified over $\mathbb{L}$ at $3$ as well. 
Therefore, we complete the proof by Theorem \ref{CFTthma}. 
\end{proof}
%\begin{prop}\label{pm}
%For any odd positive integer $m\equiv 11 \pmod{15}$, the class number of the real quadratic field $\mathbb{Q}(\sqrt{3(m^3-4)})$ is divisible by $3$.  
%\end{prop}
%\begin{proof}
%The proof is similar to that of Proposition \ref{ab}. Here, we can choose $d=4-m^3$, and $D=-3d=3(m^3-4)$. Moreover, $\alpha=2+\sqrt{4-m^3}$ works here. 
%\end{proof}
\section{Proof of Theorem \ref{thm}}
To prove that our construction can generate infinitely many quadruples of quadratic fields of a specific form with class numbers divisible by $3$, we recall a particular case of the celebrated result on integral points due to  Siegel (cf. \cite[Chapter IX, Theorem 4.3]{SI09}, \cite{SI29}). Let $V_\mathbb{Q}$ be the set of all standard absolute values on $\mathbb{Q}$. 
\begin{thma}[{\cite[Siegel's Theorem]{SI09}}]\label{thms} 
Assume that $S$ is a finite set such that $\{\infty\}\subset S\subset V_\mathbb{Q}$. Let $f(x)\in \mathbb{Q}[x]$ be a polynomial of degree at least $3$ with distinct roots in $\mathbb{C}$. Then the set $$\{(x, y)\in R_S\times R_S\mid y^2=f(x)\}$$
is finite, where $R_S=\{x\in \mathbb{Q}\mid v_p(x)\geq 0 \text{ for all } p\in V_\mathbb{Q}\setminus S\}$; so-called the ring of $S$-integers of $\mathbb{Q}$.  
\end{thma}

Given an integer $a$, we define the curve $ay^2=40500x^3+89100x^2+65340x+16215$. Assume that $$ A =\{ (x, y)\in \mathbb{Z}\times \mathbb{Z}\mid ay^2=40500x^3+89100x^2+65340x+16215\}.$$ 
If we take $S=\{\infty\}$, then $R_S=\mathbb{Z}$. Thus, by Theorem \ref{thms} the set $A$ is finite. Therefore, the set $\mathfrak{S}=\{\mathbb{Q}(\sqrt{40500x^3+89100x^2+65340x+16215})\mid x\in \mathbb{Z}^+\}$ contains infinitely many elements. Hence, by Proposition \ref{ab}, we have the following:
\begin{prop}\label{abi}
The set
$$\mathcal{S} = \{k\in \mathbb{Z}^+|\text{ the class number of } \mathbb{Q}(\sqrt{40500k^3+89100k^2+65340k+16215}) \text{ is divisible by } 3\}$$
\end{prop}
is infinite.   
\subsection*{Proof of Theorem \ref{thm}} Given a positive integer $k$, we assume that $D=40500k^3+89100k^2+65340k+16215$. Then by Proposition \ref{ab}, the class number of the real quadratic field $\mathbb{Q}(\sqrt{D})$ is divisible by $3$. 

Since $D$ is a positive integer, by Proposition \ref{ab}, the class number of the real quadratic field $\mathbb{Q}(\sqrt{40500D^3+89100D^2+65340D+16215})$ is divisible by $3$. 

Similarly, by Propositions \ref{ld} and \ref{sd}, the class numbers of the real quadratic fields  $\mathbb{Q}(\sqrt{216000D^3+457200D^2+322580D+75866})$ and  $\mathbb{Q}(\sqrt{432D^3+1080D^2+900D+223})$ are all divisible by $3$. The infinitude of such fields follows from Proposition \ref{abi}.

\section{Torsion in certain elliptic curves }

We consider the curve given by the equation,
$$y^2=40500x^3+89100x^2+65340x+16215.$$
We can verify that this defines an elliptic curve, $E_1$. 

For the next theorem, we need some preliminaries regarding the Picard group of an algebraic curve. Let $C$ be a smooth projective curve over $\bar \QQ$ and let $Div(C)$ denote the free abelian group generated by the closed points on the curve $C$. Any element is an integral linear combination of points on the curve. We denote such a combination by $D$ a Weil divisor. We say that two Weil divisors $D_1$ and $D_2$ are linearly equivalent if there exists a rational function $f$ on $C$ such that 
$$D_1-D_2=div(f).$$
Here the divisor of $f$ is the divisor defined by 
$$f^{-1}(0)-f^{-1}(\infty),$$
that is, the differences between the zeros and poles of the rational function. 

For an elliptic curve $E$ over $\bar \QQ$, the Picard group is isomorphic to 
$$Pic^0(E)\oplus \ZZ=E\oplus \ZZ\;.$$

Given a Zariski open subset $U$ inside $C$, we have the following exact sequence at the level of Picard groups:
$$\oplus_i \ZZ\to Pic(C)\to Pic(U)\to 0\;.$$
The left hand side, $\oplus_i \ZZ$ of the above exact sequence  corresponds to the free abelian group generated by the finitely many points of $C\setminus U$.

There is also the weak Mordell-Weil theorem \cite{SI09} in the context of elliptic curves. It says that the group 
$$E(\QQ)\cong \ZZ^r\oplus Tors(E(\QQ)).$$
That is $E(\QQ)$ is finitely generated.

 We want to study the torsion subgroup of $E_1$. We claim that:

{\it
 There exists a $3$-torsion in the $\QQ$-rational points on the elliptic curve $E_1$.   
}

\subsection*{Proof of the claim}
 Suppose that the elliptic curve $E_1$ does not have $3$-torsion in the group of $\QQ$-rational points $E_1(\QQ)$. Then, if we spread the elliptic curve over the rational integers and consider a fixed smooth integral model $E_{1\ZZ}$ over $\ZZ$, by the Nagell-Lutz theorem \cite[Theorem 2.5]{SI15}, the torsion points are all integral-valued. So when we specialize these torsion points in the class group of the ring of integers of the aforementioned
real quadratic field 
$$\QQ(\sqrt{40500k^3+89100k^2+65340k+16215}),$$
we obtain the torsion elements in the class group of the same order. Now by Theorem \ref{thm} there are infinitely many $k$ such that the aforementioned real quadratic fields have 3-torsions in the class group. Consider the family 
$$E_{1\ZZ}\to \mathbb{A}^1_{\ZZ},$$
where $\mathbb{A}^1_{\ZZ}$ is the affine line over $\ZZ$ given by $Spec(\ZZ[x])$. The above map is the projection map
$$(x, y)\mapsto x$$
from $E_{1\ZZ}\to \mathbb{A}^1_{\ZZ}$.
Suppose that there are no $3$-torsions in the divisor class group of $E_1$, then the same is true for the generic fiber of the family $E_{1\ZZ}\to \mathbb{A}^1_{\ZZ}$.
By generic fiber, we mean the elliptic curve $E_1$ scalar extended to $\QQ(x)$. It is an isotrivial family in the sense that all fibers are isomorphic over $\QQ$. The presence of $3$-torsion on the generic fiber forces the same on the fiber $E_1$, since the family is isotrivial.

Now, the divisor class group of $E_{1\QQ(x)}$ is the co-limit of the divisor class groups $Pic(E_{1U})$, where $E_{1U}$ is the family over a Zariski open subset $U$ of $\mathbb{A}^1_{\ZZ}$.
Then it follows that, for all the Zariski open set $U \subset \mathbb{A}^1$, $Pic(E_{1U})$ has no 3-torsion. This is because the presence of $3$-torsion in $Pic(E_{1U})$ gives the following. The $3$-torsion in $Pic(E_{1U})$ goes to zero under restriction homomorphism to $Pic(E_{1\QQ(x)})$ because the generic fiber does not have such an element. Therefore, there exists $V$ a smaller open set such that the 3-torsion in $Pic(E_{1U})$ restricted to $Pic(E_{1V})$ is zero. At the level of torsions, the map
$$Pic(E_{1U})\to Pic(E_{1V})$$
is an isomorphism, as the kernel is zero (there is no torsion element in the kernel.) This by using the localization exact sequence for Picard groups:
$$\oplus_{i}\ZZ\to Pic(E_{1U})\to Pic(E_{1V})\to 0. $$
The first term of the above sequence is the Neron-Severi components of the fibers in $E_{1U}\setminus V_{1U}$. %Hence, the claim.

But each of the class groups of
$$\QQ(\sqrt{40500k^3+89100k^2+65340k+16215})$$
has a 3-torsion subgroup for infinitely many $k$, such that the above square root is not an integer. So, the above torsion points in the class groups varies in a family in the sense that there is a Zariksi open set $U$, such that $Pic(E_{1U})$ has a $3$-torsion by \cite[Theorem 4.3]{BH}. By the above claim, there is no such $U$. So by applying the Nagell-Lutz method and using that $Pic^0(E_{1\ZZ})$ has a $3$-torsion, we find that the elliptic curve has a $3$-torsion. Thus, we our claim is proved.

The same result holds for the elliptic curves $E_2$ and $E_3$ respectively given by the equations:
$$y^2=432x^3+1080x^2+900x+223,$$
and 
$$y^2=216000x^3+457200x^2+322580x+75866$$
by arguing as above.

To sum up the above, we can state the following:
\begin{thm} Assume that $E_1$, $E_2$ and $E_3$ are as defined above. Then there is a $3$-torsion in the $\QQ$-rational points on each of $E_1$, $E_2$ and $E_3$.
\end{thm}

\subsection*{Acknowledgements}
The authors are grateful to the anonymous referees for their valuable comments that immensely improved  the presentation of the paper.
This work was supported by ANRF (SERB) Core Research Grant (CRG/2023/007323) and ANRF (SERB) MATRICS (MTR/2021/000762), Govt. of India.
%\subsection*{Competing interests declaration}
%The author declares no competing interests.

%\section*{Appendix}
%Magma computation for Proposition \ref{sd}.
%\\
%for a in [1..1000] do\\
%if a mod 15 eq 11 then\\  
%for b in [1..1000] do\\
%if b mod 45 eq 9 then\\
%if GCD(a,b)eq 1 then   \\
%$D:= 3*(4*a^3-b^2)$;\\
%C:=ClassNumber(QuadraticField(D));\\
%a, b, D, C;\\
%end if;\\
%end if;\\
%end for;\\
%end if;\\
%end for;\\
% 

\end{document}